\documentclass[12pt]{article}

\RequirePackage{isabel}
\usepackage{tikz-cd}
\usepackage{amsmath,amssymb,amsthm,amsfonts} 

\usepackage{mathrsfs}  

\usepackage{comment}
\usepackage{tikz}
\usetikzlibrary{matrix,arrows,decorations.pathmorphing}
\usepackage{rotating}
\usepackage{hyperref}
\usepackage{breqn}

\title{On the ramification of \'{e}tale cohomology groups}

\author{Isabel Leal\footnote{{Department of Mathematics, University of Chicago, 5734 S. University Avenue, 60637, Chicago, USA Email: \texttt{isabel@math.uchicago.edu} }}}

\begin{document}
\maketitle 
\begin{abstract}
	Let $K$ be a complete discrete valuation field whose residue field is perfect and of positive characteristic, let $X$ be a connected, proper scheme over $\OO_K$, and let $U$ be the complement in $X$ of a divisor with simple normal crossings.
	
	Assume  that the pair $(X,U)$ is strictly semi-stable over $\OO_K$  of relative dimension one and $K$ is of equal characteristic. We prove that, for any smooth $\ell$-adic sheaf $\GG$  on $U$ of rank one, at most tamely ramified on the generic fiber, if the ramification of $\GG$ is bounded by $t+$ for the logarithmic upper ramification groups of Abbes-Saito at points of codimension one of $X$, then the ramification of  the  \'{e}tale cohomology groups with compact support of $\GG$ is  bounded by $t+$ in the same sense.
\end{abstract}

\section{Introduction}

Let $K$ be a complete discrete valuation field with perfect residue field $k$ of characteristic $p>0$. 
Let $X$ be a connected, proper scheme 
over $\OO_K$, $D$ a divisor with simple normal crossings on $X$, and $U=X-D$.  Assume that the pair $(X,U)$ is strictly semi-stable over $\OO_K$ of relative dimension $d$ (see Definition \ref{def:semi}). 

Let  $\ell$ 	be  a prime number different from $p$ and  $\GG$ be a smooth $\ell$-adic sheaf on $U$, by which we mean a smooth $\overline{\Q_\ell}$-sheaf on $U$.  Assume that $\GG$ is at most tamely ramified on the generic fiber $X_K$. Write $D=\bigcup\limits_{i=1}^n D_i$, where $D_i$ are the irreducible components of $D$. Let $\xi_i$ be the generic point of $D_i$, $\OO_{M_i}=\OO^h_{X,\xi_i}$ the henselization of the local ring at $\xi_i$,  $M_i$ its field of fractions, and $\eta_i=\Spec M_i$. 

Let $G_{M_i}$ and $G_{K}$ denote the absolute Galois groups of $M_i$ and $K$, respectively,  and \linebreak $(G_{M_i,\,\log}^t)_{t\in\Q_{\geq 0}}$, $(G_{K,\,\log}^t)_{t\in\Q_{\geq 0}}$  the corresponding Abbes-Saito logarithmic upper ramification filtrations  (see \cite{abbessaito}). Put, for a real number $s\geq 0$, 
$G^{s+}_{M_i,\,\log}=\overline{\bigcup\limits_{t\in \Q, t>s} G_{M_i,\,\log}^t}$
and
$G^{s+}_{K,\,\log}=\overline{\bigcup\limits_{t\in \Q, t>s} G_{K,\,\log}^t}$.   Then we have the following conjecture:

\begin{conjecture} \label{conjecture}
	Under the assumptions above, if  $G_{M_i,\,\log}^{t+}$ acts trivially on  $\GG_{\overline{\eta_i}}$ for every $i$,  then  $G_{K,\,\log}^{t+}$ acts trivially on $H^j_c(U_{\overline{K}},\GG)$ for every $j$.
\end{conjecture}

Our main result is Theorem \ref{the:main}, in which 
we prove the conjecture in the special case where $\GG$ is of rank $1$, $K$ has characteristic $p$,
and the relative dimension is $d=1$.

The structure of this paper is as follows: in the first section, we shall briefly review some properties of the Abbes-Saito logarithmic upper ramification filtration and some notions on the ramification of characters. In the second section, we give a criterion for $G^{t+}_{K,\,\log}$ to act trivially on an $\ell$-adic sheaf. In the third section, we provide an  application of the Kato-Saito conductor formula. In the fourth section, we present and prove the main result. 

\section{Preliminary notions}

\subsection{The Abbes-Saito filtration}

We shall briefly review some properties of the Abbes-Saito logarithmic upper ramification filtration.
For a complete discrete valuation field $K$, Abbes and Saito constructed  a decreasing filtration  $(G_{K, \log}^t)_{t\in\Q_{>0}}$ of the absolute Galois group $G_K$,  extended by $G^0_{K,\,\log}=G_K$.

When the residue field of $K$ is perfect, $(G_{K, \log}^t)_{t\in\Q_{\geq 0}}$ coincides with the classical upper ramification filtration. $(G_{K,\,\log}^t)_{t\in\Q_{\geq 0}}$ is stable under tame base change; more precisely, if $L$ is a finite separable extension of $K$ of ramification index $e$ that is tamely ramified,   we have $G_{L,\,\log}^{et}=G_{K,\,\log}^t$. In general, for a finite separable extension $L/K$ of ramification index $e$, not necessarily tamely ramified, we have $G_{L,\,\log}^{et}\subset G_{K,\,\log}^t$.

In this paper we shall make use of the following definition:

\begin{definition}
	For a real number $s\geq 0$, define
	$G^{s+}_{K,\,\log} = \overline{\bigcup\limits_{t\in\Q, t>s}G^{t}_{K,\,\log}}$.
\end{definition}

We shall need the following property of this filtration:

\begin{lemma}[\cite{abbes2003ramification}, Lemma 5.2] \label{lemma:52}
	Let $K$ be a complete discrete valuation field with residue field $k$ of characteristic $p$. Assume that there is a map of complete discrete valuation fields $K\to L$ inducing a local homomorphism $\OO_K\to\OO_L$, that the ramification index is prime to $p$, and that the induced extension of residue fields is separable. Then, for $t\in \Q_{>0}$, the map $G_L\to G_K$ induces a surjection $G_{L,\,\log}^{et} \to G_{K,\,\log}^t$. 
\end{lemma}

As a consequence, we also have surjections $G_{L,\,\log}^{et+} \to G_{K,\,\log}^{t+}$.

\subsection{Ramification of characters}

In this subsection, assume that the residue field $k$ of $K$ has characteristic $p>0$ and is not necessarily perfect.

We recall the definition of the $k$-vector space $\Omega_k(\log)$.  There exists a canonical map $d\log:K^\times\to \Omega_k$, and  $\Omega_k(\log)$  is the amalgamate sum of the differential module $\Omega_k$ 
with $k\otimes_{\Z}K^\times$ over  $k\otimes_{\Z}\OO_K^\times$ with respect to $d\log:\OO_K^\times\to \Omega_k$ and $\OO_K^\times\hookrightarrow K^\times$. There is a residue map $\res:\Omega_k(\log)\to k$ induced by the valuation map of $K$ and an exact sequence

\begin{center}
	\begin{tikzcd}
		0 \arrow{r} & \Omega_k \arrow{r} & \Omega_k(\log) \arrow{r}{\res} & k\arrow{r} & 0.
	\end{tikzcd}
\end{center}

In \cite{kato1989swan}, Kato constructs an increasing filtration $(F_r H^1(K, \Q/\Z))_{r\in \N}$ and defines, putting $\text{Gr}_{r} H^1(K, \Q/\Z)=F_{r} H^1(K, \Q/\Z)/F_{r-1} H^1(K, \Q/\Z)$ for $r\geq1$, an injection 
\[
\rsw_{r,K}: \text{Gr}_{r} H^1(K, \Q/\Z)\to \Hom_k(\mm_K^{r}/\mm_K^{r+1},\Omega_k(\log)),
\]
where $\mm_K$ denotes the maximal ideal of $\OO_K$.  For $\chi \in F_{r} H^1(K, \Q/\Z)\backslash F_{r-1} H^1(K, \Q/\Z)$, the injection 
\[
\rsw_{r,K}(\chi):\mm_K^{r}/\mm_K^{r+1} \to \Omega_k(\log)
\]
is denoted by $\rsw_K(\chi)$ and called the refined Swan conductor of $\chi$. 

In \cite{abbes2009analyse}, Corollary 9.12, Abbes and Saito relate Kato's construction to the upper ramification groups defined in \cite{abbessaito}. More specifically, they prove that, when $K$ is of equal characteristic, 
$\chi \in F_r H^1(K, \Q/\Z)$ if and only if $\chi$ kills $G^{r+}_{K,\,\log}$.

\begin{remark}
	As the referee pointed out, the comparison between Kato's filtration and the Abbes-Saito logarithmic upper ramification groups remains open    in the mixed characteristic case. This is the only reason we assume that $K$ is of characteristic $p>0$ in sections 3 and 5 of this paper. All results of this paper are valid whenever $K$ is a complete discrete valuation field   having the property that, for all $\chi \in  H^1(K, \Q/\Z)$ and $r\in\N$,
	\[
	\chi \in F_r H^1(K, \Q/\Z) \text{ if and only if }  \chi:G_K\to \Q/\Z  \text{ kills } G^{r+}_{K,\,\log}.
	\]
\end{remark} 

Consider now the following case. Let $S=\Spec \OO_K$ and 
 $X$ 
be a regular flat separated scheme over $S$. Let $D=\bigcup\limits_{i=1}^n D_i$ be a divisor with simple normal crossings, where $D_i$ denotes the irreducible components of $D$.  For each $i$ let $\xi_i$ be a generic point for $D_i$, $\OO_{M_i}= \OO_{X, \, \xi_i}^h$ the henselization of the local ring at $\xi_i$,  $M_i$ its field of fractions, and $k_i$ the residue field of $M_i$. Let $U=X-D$ and $\chi \in H^1(U, \Q/\Z)$. For each $i$, denote by $\chi_i\in  H^1(M_i, \Q/\Z)$ the restriction of $\chi$, and by $r_i$ the Swan conductor $\mathrm{Sw}_{M_i} \chi_i$. Define the Swan divisor 
\[
D_\chi = \sum_i r_i D_i 
\]  
and let 
\[
E= \sum_{r_i>0}  D_i 
\]
be the support of $D_\chi$. It's shown by \cite{kato1989swan}, (7.3), that there exists an injection 
\[
\rsw \chi: \OO_X(-D_\chi)\otimes_{\OO_X} \OO_E\to \Omega^1_{X/S} (\log D)\otimes_{\OO_X} \OO_E       
\]
inducing $\rsw_{M_i}(\chi_i)$ at $\xi_i$. 
We say that $\chi$ is clean if $\rsw \chi$ is a locally splitting injection.

\subsection{Semi-stable pairs}

In this subsection, we let $K$ be a complete discrete valuation field with perfect residue field $k$ of characteristic $p>0$,
$X$  a  proper scheme of finite presentation over $\OO_K$, 
and $U$  an open and dense subscheme of $X$. We recall the definition of a semi-stable pair (\cite{saito2004log}, Definition 1.6):

\begin{definition}\label{def:semi}
	The pair $(X,U)$ is said to be semi-stable over $\OO_K$ of relative dimension $d$ if, \'{e}tale locally on $X$, X is \'{e}tale over $\Spec \OO_K[T_0,\ldots,T_d]/(T_0\cdots T_r-\pi)$ and $U$ is the inverse image of $\Spec \OO_K[T_0,\ldots,T_d,T_0^{-1},\ldots, T_m^{-1}]/(T_0\cdots T_r-\pi)$ for some $0\leq r\leq m\leq d$ and prime $\pi$ of $K$. 
	
	When $(X,X_K)$ is semi-stable over $\OO_K$, we say that $X$ is semi-stable over $\OO_K$. 
	
	If we substitute the condition ``\'{e}tale locally'' by ``Zariski locally'', the pair $(X,U)$ is then said to be strictly semi-stable.  
\end{definition}

We shall need the following property of semi-stable pairs, which is a consequence of Theorem 2.9 in \cite{saito2004log}:
\begin{theorem}\label{theo:semi}
	Let $(X,U)$ be a strictly semi-stable pair over $\OO_K$ and $L$ be a finite separable extension of $K$. Then there exists a proper birational morphism $X'\to X_{\OO_L}$ inducing an isomorphism $U'\to U_{\OO_L}$, where $U'$ is the inverse image of $U_{\OO_L}$, and such that $(X', U')$ form a strictly semi-stable pair over $\OO_L$.
\end{theorem}
\section{\texorpdfstring{The action of  $G^{t+}_{K, \, \log}$}{The action of G^{t+}_{K, log}}}

In this section, we let $K$ be a complete discrete valuation field of equal characteristic with perfect residue field $k$ of characteristic $p>0$, $\ell$ be a prime different than $p$, and 
$M, N$ be finite-dimensional representations of $G_K$ over $\overline{\Q_\ell}$ which come from  finite-dimensional continuous representations of $G_K$ over a finite extension of $\Q_\ell$ contained in $\overline{\Q_\ell}$. 
We shall provide a criterion for $G^{t+}_{K,\,\log}$ to act trivially on $M$.

There is   a canonical slope decomposition (see \cite{katz1988gauss}, Proposition 1.1, or \cite{abbes2011ramification}, Lemma 6.4) 
\[
M= \bigoplus_{r\in\Q_{\geq0}} M^{(r)}
\]
characterized by the following properties: if $P$ is the wild inertia subgroup of $G_K$, then $M^P=M^{(0)}$. Further, for all $r>0$, 
\[
(M^{(r)})^{G_{K,\,\log}^r} = 0
\]
and
\[
(M^{(r)})^{G_{K,\,\log}^{r+}} = M^{(r)}.
\]
We have $M^{(r)}=0$ except for finitely many $r$. The values of $r$ for which 
$M^{(r)}\neq 0$ are called slopes of $M$. 
\begin{definition}
	We say that M is isoclinic if it has only one slope.
\end{definition}

The following proposition gives our criterion:
\begin{proposition}\label{prop:action}Let $t$ be a nonnegative real number.
	Assume that, for any  totally tamely ramified extension $L/K$ of degree $e$ prime to $p$, we have the following: if $M_L$ denotes the  representation of $G_L$ induced by $M$, then, for any character $\chi:G_L\to \overline{\Q_\ell}^\times$ for which $\mathrm{Sw}_{L}(\chi)>et$, we have
	\[
	\mathrm{Sw}_{L}(M_L\otimes\chi)=\rk(M_L)\mathrm{Sw}_L(\chi).
	\]
	Then ${G_K^{t+}}$ acts trivially on $M$.
\end{proposition}
The proof will be presented shortly.  The general strategy is the following:
\begin{itemize}
	\item[\textendash] We first show that the behavior of the tensor product of isoclinic $M$ and $N$ is similar to that of the tensor product of characters;
	\item[\textendash] Next, we use the previous result to understand the slope decomposition of the tensor product $M\otimes\chi$ and prove the proposition.
\end{itemize}

We start with the lemma:

\begin{lemma}\label{lemma:tensoriso}
	If $M$ is isoclinic of slope $r$ and $N$ is isoclinic of slope $s$, where $r>s$, then $M\otimes N$ is isoclinic of slope $r$. 
\end{lemma}
\begin{proof}
	We have
	\[
	M^{G_K^r} = 0,
	\]
	\[
	M^{G_K^{r+}} = M,
	\]
	\[
	N^{G_K^s} = 0,
	\]
	and
	\[
	N^{G_K^{s+}} = N.
	\]
	Since $r>s$, $(M\otimes N)^{G_K^{r+}} = M\otimes N$. On the other hand, $G_K^{r}$ acts trivially on $N$ and $M^{G_K^r} = 0$, so $(M\otimes N)^{G_K^r} = 0$. Hence $M\otimes N$ is isoclinic of slope $r$.
\end{proof}

\begin{proof}[Proof of Proposition \ref{prop:action}]
	We need to show that, if $r>t$, then $M^{(r)}=0$. Let $R$ be the maximum slope of $M$.  Assume, by contradiction, that $R>t$. Let $m, e$ be positive integers such that:
	\begin{enumerate}[(i)]
		\item $e$ is prime to $p$,
		\item $\frac{m}{e}<R$,
		\item $\frac{m}{e}$ is strictly greater than any other slope of $M$,
		\item $\frac{m}{e}>t$. 
	\end{enumerate}
	
	Let $L$ be a  totally tamely ramified extension of degree $e$ of $K$. By \cite{abbessaito}, Proposition 3.15, $G_{K,\,\log}^s=G_{L,\,\log}^{es}$ for any $s\in\Q_{\geq 0}$, so  
	the slopes of $M_L$ are of the form $er$, where $r$ is a slope of $M$.
	Take $\chi$ with $\mathrm{Sw}_L(\chi)=m$. Then, by assumption,
	\[
	\mathrm{Sw}_{L}(M_L\otimes\chi)=\rk(M_L)\mathrm{Sw}_L(\chi)=\rk(M_L)m.
	\]
	By Lemma \ref{lemma:tensoriso}, for all $r<m$ we have that $M_L^{(r)}\otimes \chi$ is isoclinic of slope $m$, while $M_L^{(eR)}\otimes \chi$ is isoclinic of slope $eR$. It follows that  
	\[
	\mathrm{Sw}_{L}(M_L\otimes\chi)= \sum_{r\in\Q_{\geq 0}}\mathrm{Sw}_{L}(M_L^{(r)}\otimes\chi)= \sum_{r\in\Q_{\geq 0}, r<m} \rk(M_L^{(r)}) m +  \rk (M_L^{(eR)})eR.
	\]
	Combining the two expressions we get
	\[
	\rk (M_L^{(eR)})eR =  \rk (M_L^{(eR)})m,
	\]
	which is a contradiction, since, by assumption, $m<eR$ and $M_L^{(eR)}\neq 0$. 
\end{proof}

\section{The Kato-Saito conductor formula}

Let $K$ be a complete discrete valuation field with perfect residue field $k$ of characteristic $p>0$. Let $\ell$ be a prime number different from $p$, $U$ be a smooth separated scheme of finite type over $K$, and $\FF$ be a smooth $\ell$-adic sheaf of constant rank on $U$. In \cite{kato2013ramification},  Kato and  Saito defined the Swan class $\mathrm{Sw}_U\FF$, a $0$-cycle class with coefficients in $\Q$ supported on the special fiber of a compactification of $U$ over $\OO_K$, and proved the conductor formula
\[
\mathrm{Sw}_{K} R\Gamma_c(U_{\overline{K}},\FF) = \deg \mathrm{Sw}_U\FF + \rk(\FF) \mathrm{Sw}_{K} R\Gamma_c(U_{\overline{K}},\overline{\Q_\ell}),
\]   
where $\mathrm{Sw}_{K} R\Gamma_c(U_{\overline{K}},\FF)$ denotes the alternating sum  
$\sum\limits_j (-1)^j \mathrm{Sw}_{K} H^j_c(U_{\overline{K}}, \FF)$.

In this section, assume that 
$X$ is a regular flat separated scheme of finite type over $S=\Spec \OO_K$. Let $D\subset X$ be a divisor with simple normal crossings and write $D=\bigcup\limits_{i=1}^n D_i$, where $D_i$ are the irreducible components of $D$.    Put $U=X-D$ and consider a  smooth $\ell$-adic sheaf $\FF$ of  rank $1$ on $U$, at most tamely ramified on $X_K$ and with clean ramification with respect to $X$. 

The Swan $0$-cycle class $c_\FF$ of $\FF$ is  defined as follows. Let $E$ be the support of the Swan divisor $D_\FF= \sum r_i D_i$. Then define $c_\FF\in CH_0(E)$ as
\[
c_\FF = \{ c(\Omega^1_{X/S} (\log D) \otimes_{\OO_X} \OO_E)^* \cap (1+D_\FF)^{-1}\cap D_\FF \}_{\dim 0}.
\]

Under the assumption that $\dim U_K\leq 1$,  by Corollary 8.3.8 of \cite{kato2013ramification}, the Kato-Saito conductor formula becomes simply 
\[
\mathrm{Sw}_{K} R\Gamma_c(U_{\overline{K}},\FF) = \deg c_\FF + \mathrm{Sw}_{K} R\Gamma_c(U_{\overline{K}},\overline{\Q_\ell}).
\]

The following proposition is an application of this formula that will be useful in the next section:

\begin{proposition}\label{prop:tensor}
	Let $X$, $S$ and $U=X-D$ be as above. Let $\FF_1$ and $\FF_2$ be two 
	smooth $\ell$-adic sheaves on $U$ of rank one, $\FF_2$ having clean ramification with respect to $X$. 
	Write $D_{\FF_1}= \sum r_i D_i$ and $D_{\FF_2}=\sum s_i D_i$. Assume that $r_i<s_i$ for every $i$. Then $\FF_1\otimes\FF_2$ has clean ramification and
	\[
	c_{\FF_1\otimes\FF_2}= c_{\FF_2}.
	\]
\end{proposition} 
\begin{proof} 
	Since $r_i< s_i$ for every $i$, we have $D_{\FF_1\otimes\FF_2}=D_{\FF_2}$ and the refined Swan conductors of  $\FF_1\otimes\FF_2$ and $\FF_2$ coincide. Denote by $E_i$ the support of $D_{\FF_i}$ and by $E$ be the support of $D_{\FF_1\otimes\FF_2}$. We have $E=E_2$, so
	\begin{align*}
	c_{\FF_1\otimes\FF_2}&= \{ c(\Omega^1_{X/S} (\log D) \otimes_{\OO_X} \OO_E)^* \cap (1+D_{\FF_1\otimes\FF_2})^{-1}\cap D_{\FF_1\otimes\FF_2}\}_{\dim 0}
	\\ 
	&=\{ c(\Omega^1_{X/S} (\log D) \otimes_{\OO_X} \OO_{E_2})^* \cap (1+D_{\FF_2})^{-1}\cap D_{\FF_2}\}_{\dim 0}\\
	&= c_{\FF_2}.
	\qedhere\end{align*}
\end{proof}

\section{Main results}
In this section, we let $K$ be a complete discrete valuation field with perfect residue field $k$ of characteristic $p>0$ and of equal characteristic, 
$S=\Spec \OO_K$, and $s=\Spec k$.
We will denote by $X$ a  proper, connected scheme of finite presentation over $\OO_K$, 
and $U$  an open and dense subscheme of $X$. We assume that $D=X-U$ is a  
divisor with simple normal crossings and write 
$\bigcup\limits_{i=1}^n D_i$, where $D_i$ are the irreducible components. We also assume  that the pair $(X,U)$ is strictly semi-stable over $\OO_K$ of relative dimension $1$, and that $\GG$ is a smooth $\ell$-adic sheaf on $U$, where  $\ell$ 	is  a prime number different from $p$.  Further, we assume that $\GG$ is of rank $1$ and at most tamely ramified on the generic fiber $X_K$.
Denote by $\xi_i$  the generic point of $D_i$, $\OO_{M_i}=\OO^h_{X,\xi_i}$ the henselization of the local ring at $\xi_i$,  $M_i$ its field of fractions, $k_i$ the residue field of $M_i$, and $\eta_i=\Spec M_i$. 

We shall prove the following theorem:
\begin{theorem} \label{the:main}
	Conjecture \ref{conjecture} is true when $\GG$ is of rank $1$, the relative dimension is $1$, and $K$ is of equal characteristic. 
\end{theorem}

\begin{remark}
	When the relative dimension is greater than $1$, one should still be able to prove Conjecture \ref{conjecture} using the same methods used in this paper,  as long as it is true that \[\mathrm{Sw}_{K} R\Gamma_c(U_{\overline{K}},\FF) = \deg c_\FF + \mathrm{Sw}_{K} R\Gamma_c(U_{\overline{K}},\overline{\Q_\ell})\] for   smooth $\ell$-adic sheaves $\FF$ of  rank $1$ on $U$, at most tamely ramified on $X_K$ and with clean ramification with respect to $X$. 
\end{remark}

The proof is divided in two cases. First observe that, since the total constant field of $X_K$ is a finite unramified extension of $K$, we may assume that $K$ is the total constant field of $X_K$.
Then there is an exact sequence of fundamental groups
\begin{center}
	\begin{tikzcd}
		1 \arrow{r} & \pi_1(U_{\overline{K}})\arrow{r} & \pi_1(U) \arrow{r} & G_K \arrow{r} & 1.
	\end{tikzcd}
\end{center}

Let $M$ be the function field of $X$ and $\eta=\Spec  M$. We first consider the case in which the action of $\pi_1(U_{\overline{K}})$ is trivial on $\GG_{\bar{\eta}}$, and then the case in which it is non-trivial. 

To prove the first case, we shall need the following lemma:

\begin{lemma}\label{lemma:forcurve}
	In addition to the assumptions of Theorem \ref{the:main}, assume that $\GG_{\bar \eta}$ is the pullback of some $\ell$-adic representation $\HH$ of $G_K$. If
	$G_{M_i,\,\log}^{t+}$ acts trivially on $\GG_{\overline{\eta_i}}$, then $G_{K}^{t+}$ acts trivially on $\HH$.
\end{lemma}
\begin{proof}
	This follows from Lemma \ref{lemma:52}.
\end{proof}

\begin{proposition}\label{prop:main}
	Theorem \ref{the:main} holds if $\pi_1(U_{\overline{K}})$ acts trivially on $\GG_{\bar \eta}$.
\end{proposition}
\begin{proof}
	In this case, by the homotopy exact sequence of \'{e}tale fundamental groups, we have that $\GG_{\bar \eta}$ is the pullback of some $\ell$-adic representation $\HH$ of $G_K$. Then
	\[ H_c^j(U_{\overline{K}},\GG) = H_c^j(U_{\overline{K}},\overline{\Q_{\ell}})\otimes \HH.\]
	By Lemma  \ref{lemma:tensoriso}, and the fact that  $H_c^j(U_{\overline{K}},\overline{\Q_{\ell}})$
	is at most tamely ramified (\cite[Corollary 2]{Saito1993}), 
	we have that the slope decomposition of $H_c^j(U_{\overline{K}},\GG)$ coincides with that of $\HH$, in the following sense:
	\[
	(H_c^j(U_{\overline{K}},\GG))^{(r)} = H_c^j(U_{\overline{K}},\overline{\Q_{\ell}})\otimes \HH^{(r)}.
	\]
	It follows that $G_{K,\,\log}^t$ acts trivially on $H_c^j(U_{\overline{K}},\GG)$ if and only if it acts trivially on $\HH$. By Lemma \ref{lemma:forcurve}, the result follows. 
\end{proof}

We shall now prove Theorem \ref{the:main} for the case in which
$\pi_1(U_{\overline{K}})$ does not  act trivially on $\GG_{\bar \eta}$. The core of strategy is the following:  using the Kato-Saito conductor formula and the fact that $H^0_c(U_{\overline{K}},\GG)=H^2_c(U_{\overline{K}},\GG)=0$, we show  that $H^1_c(U_{\overline{K}},\GG)$ satisfies the hypotheses of Proposition \ref{prop:action}. 

\begin{lemma}
	\label{lemma:2}
	Keep the assumptions of Theorem \ref{the:main}. Let $e$ be a natural number prime to $p$ and $L$ be a  totally tamely ramified extension of $K$ of degree $e$. If  $\chi:G_{L}\to \overline{\Q_{\ell}}^\times$ is a character such that $\mathrm{Sw}_{L}(\chi)>et$, then 
	\[
	\mathrm{Sw}_{L}( R\Gamma_c(U_{\overline{L}},\GG)\otimes \chi)=
	\rk( R\Gamma_c(U_{\overline{L}},\GG))\mathrm{Sw}_{L}(\chi).
	\]
\end{lemma}
\begin{proof}
	
	First consider the following. By Theorem \ref{theo:semi}, 
	there exists a proper birational morphism $X'\to X_{\OO_L}$ inducing an isomorphism $U'\to U_{\OO_L}$, where $U'$ is the inverse image of $U_{\OO_L}$, and such that $(X',U')$ is strictly semi-stable over $\OO_L$.
	
	Let $D'=X'-U'$ and write  $D'=\bigcup\limits_{i=1}^{n'}D_i'$, where $D_i'$ are the irreducible components of $D'$. For each $1\leq i\leq n'$ let 
	$\xi_i'$ be the generic point of $D_i'$, $\OO_{M_i'}=\OO^h_{X', \, \xi_i'}$ the henselization of the local ring at $\xi_i'$, $M_i'$ its field of fractions, and $\eta_i'=\Spec M_i'$. 
		
	There is a composition of blowups of closed points $\tilde{X}\to X$ and a point $\tilde{\xi}_i$ such that $\OO_{\tilde{X},\, \tilde{\xi}_i}=\OO_{X', \, \xi'_i}  \cap M$. Let $\tilde{M}_i$ be the field of fractions of  $\OO^h_{\tilde{X}, \, \tilde{\xi}_i}$.    Put $\tilde{\eta}_i=\Spec \tilde{M}_i$.
	Denote by $e_i'$ and $\tilde{e}_i$ the ramification indices  of $M_i'/\tilde{M}_i$ and $\tilde{M}_i/K$, respectively. We have $e=e_i'\tilde{e}_i$. 
	
	By \cite{kato1989swan}, Theorem 8.1, and the fact that $G_{M_{i},\,\log}^{t+}$ acts trivially on 
	$\GG_{\overline{\eta_{i}}}$ for every $1\leq i \leq n$, we have that $G_{\tilde{M}_{i},\,\log}^{\tilde{e}_it+}$ acts trivially on 
	$\GG_{\overline{\tilde{\eta}_{i}}}$ for every $1\leq i \leq n'$. Further, 
	since we have $G_{M_{i}',\,\log}^{e_i'\tilde{e}_i t+}\subset G_{\tilde{M}_{i},\,\log}^{\tilde{e}_i t+}$, we get that
	$G_{M_{i}',\,\log}^{et+}$ acts trivially on 
	$\GG_{\overline{\eta_{i}'}}$ for all $1\leq i \leq n'$. 
	Thus it is enough to prove that
	\[
	\mathrm{Sw}_{K}( R\Gamma_c(U_{\overline{K}},\GG)\otimes \chi)=
	\rk( R\Gamma_c(U_{\overline{K}},\GG))\mathrm{Sw}_{K}(\chi)
	\]
	for  $\chi:G_{K}\to \overline{\Q_{\ell}}^\times$  such that $\mathrm{Sw}_{K}(\chi)>t$.
	
	Put $r=\mathrm{Sw}_{K}(\chi)$ and
	denote by $\tilde{\chi}$ 
	the pullback of $\chi$ to $U$. $\tilde \chi$ has clean ramification because the following diagram
	\[
	\begin{tikzcd}[row sep = 0.5 in, column sep = 0.5 in]
	\mm_K^r/\mm_K^{r+1} \arrow{r}{\rsw_K\chi} \arrow{d} & \Omega_k(\log)  \arrow[hookrightarrow]{d} \\
	\mm_{M_i}^r/\mm_{M_i}^{r+1} \arrow{r}{\rsw_{M_i}\tilde\chi} &  \Omega_{k_i}(\log)
	\end{tikzcd}
	\]
	is commutative. Indeed, since $\chi$ is clean and $\Omega_k(\log)\hookrightarrow\Omega_{k_i}(\log)$ is  a splitting injection, $\rsw\tilde\chi$ is a locally  splitting injection.  Further, by Lemma \ref{lemma:52}, $\mathrm{Sw}_{M_i}(\tilde{\chi})>t$ for every $i$. From the Kato-Saito conductor formula, Proposition \ref{prop:tensor}, and the fact that $(X,U)$ is semi-stable over $\OO_K$, we have that $\GG\otimes\tilde{\chi}$ is clean and
	\[\mathrm{Sw}_{K} R\Gamma_c(U_{\overline{K}},\GG\otimes \tilde{\chi}) = \deg c_{\GG\otimes \tilde{\chi}}=\deg  c_{\tilde{\chi}}.\]
	
	Again by the Kato-Saito conductor formula,
	\[
	\mathrm{Sw}_{K} R\Gamma_c(U_{\overline{K}}, \tilde{\chi}) = \deg c_{\tilde{\chi}}.
	\]
	Therefore, we have
	\[
	\mathrm{Sw}_{K} R\Gamma_c(U_{\overline{K}},\GG\otimes \tilde{\chi})=\mathrm{Sw}_{K} R\Gamma_c(U_{\overline{K}}, \tilde{\chi})
	=\mathrm{Sw}_{K} (R\Gamma_c(U_{\overline{K}},\overline{\Q_\ell})\otimes \chi).
	\]
	Since
	\[
	\mathrm{Sw}_{K} R\Gamma_c(U_{\overline{K}},\GG\otimes \tilde{\chi})=
	\mathrm{Sw}_{K}( R\Gamma_c(U_{\overline{K}},\GG)\otimes \chi)
	\]
	and
	\[
	\mathrm{Sw}_{K} (R\Gamma_c(U_{\overline{K}},\overline{\Q_\ell})\otimes \chi)=
	\rk( R\Gamma_c(U_{\overline{K}},\overline{\Q_\ell}))\mathrm{Sw}_{K}(\chi)= \rk( R\Gamma_c(U_{\overline{K}},\GG))\mathrm{Sw}_{K}(\chi),
	\]
	we conclude that 
	\[
	\mathrm{Sw}_{K}( R\Gamma_c(U_{\overline{K}},\GG)\otimes \chi)=
	\rk( R\Gamma_c(U_{\overline{K}},\GG))\mathrm{Sw}_{K}(\chi). \qedhere
	\]
\end{proof}

\begin{lemma}\label{lemma:3}
	Let the assumptions be the same as in Lemma \ref{lemma:2}, and assume further that $\pi_1(U_{\overline{K}})$ does not act trivially on $\GG$. Then
	\[
	H^j_c(U_{\overline{L}},\GG)=0
	\]
	for every $j\neq 1$.
\end{lemma}

\begin{proof}
	By Poincar\'{e} duality and the fact that $X$ is of dimension $2$, it's enough to show that 
	$H^0(U_{\overline{L}},\GG)=0$. Since $\pi_1(U_{\overline{K}})$ does not act trivially on $\GG_{\bar{\eta}}$ and $\rk (\GG)=1$, we get that $H^0(U_{\overline{L}},\GG)=\GG_{\bar{\eta}}^{\pi(U_{\overline{L}})}=0$.
\end{proof}

\begin{proof}[Proof of Theorem \ref{the:main}]
	The theorem has already been proved in Proposition \ref{prop:main} for $\GG$ such that $\pi_1(U_{\overline{K}})$ acts trivially on it, so we assume that $\pi_1(U_{\overline{K}})$ does not act trivially.   By Lemma \ref{lemma:3}, it's enough to prove that  $G_{{K},\,\log}^{t+}$ acts trivially on $H^1_c(U_{\overline{K}},\GG)$. 
	
	From Lemmas \ref{lemma:2} and \ref{lemma:3}, it follows that
	\[
	\mathrm{Sw}_{L} (H^1_c(U_{\overline{L}},\GG)\otimes \chi)= \rk(H^1_c(U_{\overline{L}},\GG))\mathrm{Sw}_L(\chi)
	\]
	for any  totally tamely ramified extension $L$ of $K$ of degree $e$ prime to $p$ and arbitrary character $\chi: G_{L}\to \overline{\Q_\ell}^\times$ satisfying $\mathrm{Sw}_L(\chi)>et$. 
	
	From Proposition \ref{prop:action}, we have that $G_{K,\,\log}^{t+}$ acts trivially on $H^1_c(U_{\overline{K}},\GG)$. Hence $G_{K,\,\log}^{t+}$ acts trivially on $H^j_c(U_{\overline{K}},\GG)$ for every $j$. 
\end{proof}

\vspace{2 cm}
\begin{ack}
	My sincere gratitude to my advisor, Professor Kazuya Kato, who kindly provided me with his invaluable advice, guidance and feedback throughout the elaboration of this paper.
	
	I would also like to thank the anonymous referee for the important comments and suggestions that helped improve this work. 
\end{ack}

\bibliographystyle{plain}
\bibliography{bibfile}

\end{document}